\documentclass[12pt,a4paper]{article}

\usepackage{amsfonts,amssymb,amsthm}
\usepackage{amsbsy}
\usepackage{graphicx}
\usepackage{latexsym}
\usepackage{amsmath,enumerate,amsfonts,amssymb,color,graphicx,amsthm}

\usepackage{CJK,CJKnumb}

\theoremstyle{plain}
\newtheorem{theorem}{Theorem}[section]

\newtheorem{lemma}[theorem]{Lemma}

\theoremstyle{definition}

\theoremstyle{remark}
\newtheorem{remark}{Remark}[section]
\numberwithin{equation}{section}

\theoremstyle{example}

\newcounter{marnote}

\def\o{\overline}

\def\ma{Monge-Amp\`{e}re }
\def\disp{\displaystyle} 
\def\b{\backslash}\def\dint{\displaystyle\int}

\begin{document}

\begin{center}{\Large Necessary and sufficient conditions on existence of radial solutions for exterior Dirichlet problem of fully nonlinear elliptic equations}
\end{center}
\centerline{\large  Limei Dai \quad Jiguang Bao \quad Bo Wang }
\vspace{5mm}

\begin{minipage}{140mm}{\footnotesize {\small\bf Abstract:} {\small
In this paper, we study the exterior Dirichlet problem for the fully nonlinear elliptic equation $f(\lambda(D^{2}u))=1$. We obtain the necessary and sufficient conditions of existence of radial solutions with prescribed asymptotic behavior at infinity.}}

 {\small{\bf Keywords:} fully nonlinear elliptic equation; radial solutions; necessary and sufficient conditions; exterior Dirichlet problem; asymptotic behavior}

{\small{\bf 2010 MSC.} 35J60; 35B40; 35B07}

\end{minipage}

\section{Introduction}

In this paper, we study the exterior Dirichlet problem of the fully nonlinear elliptic equation
\begin{align}
\label{ns1}f(\lambda(D^2u))=&1\ \ \mbox{in}\ \ \mathbb{R}^n\b\o{B_1(0)},\\
\label{ns2}u=&b\ \ \mbox{on}\ \ \partial B_1(0).
\end{align}
where $B_1(0)=\{x\in \mathbb{R}^n:|x|<1\}$ is the unit ball in $\mathbb{R}^n$, $b$ is a constant, and $f(\lambda)$ is a given symmetric function of the eigenvalues $\lambda=(\lambda_1,\dots,\lambda_n)$ of the Hessian matrix $D^2u$. We study $f$ in an open convex symmetric cone $\Gamma\subset \mathbb{R}^n$ with vertex at the origin,
\begin{equation}\label{Gamma}\{\lambda\in \mathbb{R}^n|\lambda_j>0,1\leq j\leq n\}\subset\Gamma\subset\{\lambda\in \mathbb{R}^n|\sum_{j=1}^{n}\lambda_j>0,1\leq j\leq n\}.\end{equation}
Suppose that $f\in C^{\infty}(\Gamma)\cap C^0(\o{\Gamma})$ is concave and symmetric in $\lambda_j$,
\begin{equation}\label{f}f>0\ \mbox{in}\ \Gamma,\ f=0\ \mbox{on}\ \partial\Gamma;\ f_{\lambda_j}>0\ \mbox{in}\ \Gamma\ \forall 1\leq j\leq n.\end{equation}
Let $\sigma_k(\lambda)=\sum_{1\leq
i_1<\cdots<i_k\leq
n}\lambda_{i_1}\cdots\lambda_{i_k},k=1,2,\cdots,n,$ and $\Gamma_k=\{\lambda\in {\bf{\mathbb{R}}}^n|\sigma_j(\lambda)>0,j=1,2,\cdots,k\}$. Then $(f,\Gamma)=(\sigma_k^{\frac{1}{k}},\Gamma_k)$ and $(f,\Gamma)=((\frac{\sigma_k}{\sigma_l})^{\frac{1}{k-l}},\Gamma_k), 1\leq l<k\leq n$ are the special cases of $(f,\Gamma).$ In particular, if $k=n$, $(f,\Gamma)=(\sigma_k^{\frac{1}{k}},\Gamma_k)$ corresponds to the \ma operator.

%For the equation \eqref{ns1} in an interior bounded domain, there are many excellent studies about the existence of solutions \cite{CNSIII}

A classical theorem for \ma equation states that any classical convex solution of
\begin{equation*}\label{ma}\det D^2u=1\ \mbox{in}\ \mathbb{R}^n\end{equation*}
must be a quadratic polynomial. This theorem was established by J\"{o}rgens \cite{j} ($n=2$), Calabi \cite{c} ($n\leq 5$) and Pogorelov \cite{p} ($n\geq 2$). Later Cheng-Yau \cite{cy} proved the J\"{o}rgens-Calabi-Pogorelov theorem by the simpler and more analytical way along the lines of affine geometry. This result for viscosity solutions was extended by Caffarelli \cite{ca}. Jost-Xin \cite{jx} also gave another proof of this theorem. However, Trudinger-Wang \cite{tw1} proved that if $D$ is an open convex subset in $\mathbb{R}^n$ and $u$ is a convex $C^2$ solution to $\det D^2u=1$ in $D$ with $\lim_{x\to\partial D}u(x)=\infty,$ then $D=\mathbb{R}^n$.

In 2003, Caffarelli-Li \cite{cl} made an extension of the J\"{o}rgens-Calabi-Pogorelov theorem to exterior domains. Moreover, Caffarelli-Li \cite{cl} also established the existence of solutions with asymptotic behavior at infinity to the exterior Dirichlet problem of \ma equations.
 \begin{theorem}\label{thmex}(\cite{cl}) Let $\Omega$ be a smooth, bounded, strictly convex domain in $\mathbb{R}^n, n\geq 3$ and $\phi\in C^2(\partial \Omega)$. Then for any symmetric positive definite matrix $A\in\mathbb{R}^{n\times n}$, $\hat{b}\in\mathbb{R}^{n}$, there exists a constant $c_1=c_1(n,\Omega,\phi,\hat{b},A)$ such that for any $\hat{c}>c_1$, there exists a unique solution $u\in C^{\infty}(\mathbb{R}^n\b\o{\Omega})\cap C^0(\mathbb{R}^n\b\Omega)$ satisfying
 $$\begin{cases}
 \det D^2u=1\ \mbox{in}\ \mathbb{R}^n\b\o{\Omega},\\
 u=\phi\ \mbox{on}\ \partial\Omega,\\
 \disp\lim_{|x|\to\infty}|x|^{n-2}\left|u(x)-\left(\frac{1}{2}x^{T}Ax+\hat{b}\cdot x+\hat{c}\right)\right|=0.
 \end{cases}$$
\end{theorem}
For $n=2$, the existence of solutions to the exterior Dirichlet problem for \ma equation $\det D^2u=1$ was established by Bao-Li \cite{bl-1} using the Perron method. Ferrer-Mart\'{i}nez-Mil\'{a}n \cite{FMM1, FMM2} also studied the similar problems by using the complex variable method. We can also refer to Delano\"{e} \cite{De}. Bao-Li-Zhang \cite{blz} proved the existence of solutions to the exterior Dirichlet problem for $\det D^2u=f$ with $f$ being a perturbation of $1$ at infinity for $n\geq 2$. Ju-Bao \cite{jb} obtained the existence of exterior solutions with $f$ being a preturbation of $f_0(|x|)$ at infinity for $n\geq 3$. For the fully nonlinear elliptic equations \eqref{ns1}, Li-Bao \cite{bl-2} obtained the existence of solutions of the exterior Dirichlet problem.

The constant $c_1$ in Theorem \ref{thmex} plays an important role in the existence and nonexistence of solutions to the exterior Dirichlet problem. Wang-Bao \cite{wb} first studied the constant among the radially symmetric solutions to Hessian equations $\sigma_k(\lambda(D^2u))=1$.
\begin{theorem}\label{thmns}(\cite{wb}) Let $n\geq 3, 2\leq k\leq n$ and $a=(\frac{1}{C_n^k})^{1/k}$. There exists a unique radially symmetric solution $u\in C^{2}(\mathbb{R}^n\b\o{B_1(0)})\cap C^{1}(\mathbb{R}^n\b B_1(0))$ satisfying
\begin{alignat}{2}
\label{r1}\sigma_k(\lambda(D^2u))=&1\ \ \mbox{in}~\mathbb{R}^{n}\b\o{B_1(0)},\\
\label{r2} u=&\hat{a}\ \ \mbox{on}~\partial B_1(0),\\
\label{r3} u=&\frac{a}{2}|x|^2+\hat{c}+O(|x|^{2-n})\ \ \mbox{as}~|x|\to\infty,
\end{alignat}
if and only if $\hat{c}\geq C^*=\hat{a}-\frac{a}{2}+a\int_1^{\infty}s((1-s^{-n})^{\frac{1}{k}}-1)ds.$
 \end{theorem}
 Recently, Li-Lu \cite{ll} characterized the existence and nonexistence of solutions in terms of the asymptotic behavior to the exterior Dririchlet problem with the right hand side being $1$ or the perturbation of $1$ at infinity. In this paper, we obtain the necessary and sufficient conditions of existence of radial solutions to \eqref{ns1} and \eqref{ns2}. We suppose that
 \begin{equation}\label{c*}\mbox{there exists a constant}\ c_*\ \mbox{such that} \ f(c_*,\dots,c_*)=1.\end{equation}
 Let $\Omega$ be a domain in $\mathbb{R}^n$. A function $u\in C^2(\Omega)$ is called admissible  if at each $x\in \Omega$, $\lambda(D^2u(x))\in \Gamma.$ The condition \eqref{f} guarantees that \eqref{ns1} is elliptic for the admissible functions. Set
 $$\Phi=\{u\in C^1(\mathbb{R}^n\b B_1)\cap C^2(\mathbb{R}^n\b \o{B_1})|u\ \mbox{is an admissible radially symmetric function.}\}$$

The main result is the following.
\begin{theorem}\label{thmns1} Let $n\geq 3$ and \eqref{c*} hold. There exists a unique function $u\in \Phi$ satisfying \eqref{ns1}, \eqref{ns2} and
\begin{equation}\label{ns3}u(x)=\dfrac{c_*}{2}|x|^2+c+O(|x|^{2-n})\end{equation}
if and only if $c\in [c_0,+\infty)$ where $c_0=\mu(W^{-1}_s(\gamma_0))+b-c_*/2$, $0\leq \gamma_0<c_*$, \begin{equation}\label{mualpha}\mu(\alpha)=\dint_1^{+\infty}s(W(s,\alpha)-c_*)ds\end{equation}
and $W_s(\alpha)=W(s,\alpha)$ satisfies \eqref{W} and \eqref{Wboun}.
\end{theorem}

\section{Proof of Theorem \ref{thmns1}}

\begin{lemma}(\cite{wb})\label{lem-w}
Let $\lambda=(\beta,\gamma,\dots,\gamma)\in \Gamma.$ Then $\gamma>0.$
\end{lemma}

\begin{lemma}\label{lemma1}
Let $\lambda=(\beta,\gamma,\dots,\gamma)\in \mathbb{R}^{n}, n\geq 2$ and $f(\lambda)=1.$ Then $\lambda\in \Gamma$ if and only if there exists a constant $0\leq\gamma_0<c_*$ such that \begin{equation}\label{gamma}\gamma_0<\gamma<+\infty.\end{equation}
\end{lemma}

\begin{proof}
Consider the equation
\begin{equation}\label{f-equ}
f(\beta,\gamma,\dots,\gamma)=1\ \mbox{for}\ (\beta,\gamma,\dots,\gamma)\in \Gamma.
\end{equation}
Similar to \cite{bl-2}, from \eqref{f} and \eqref{c*}, we know that
$$f(c_*,\gamma,\dots,\gamma)>1,\ \mbox{if} \ \gamma>c_*.$$
By \eqref{f}, for $(\beta_0,\gamma,\dots,\gamma)\in \partial \Gamma$,
$$f(\beta_0,\gamma,\dots,\gamma)=0.$$
Hence from the intermediate value theorem and \eqref{f}, for each $\gamma>c_*$, there exists a unique $g(\gamma)$ between $\beta_0$ and $c_*$ such that $(g(\gamma),\gamma,\dots,\gamma)\in \Gamma,$ and
\begin{equation}\label{f-g}f(g(\gamma),\gamma,\dots,\gamma)=1,\ \forall\ \gamma>c_*.\end{equation}
Then we can define the continuous and differentiable function $g$ such that
\begin{equation*}\beta=g(\gamma),\ \gamma\geq c_*\end{equation*}
and $$g(c_*)=c_*.$$
Moreover, differentiating \eqref{f-g} with respect to $\gamma$, we get
\begin{equation}\label{g2prime}f_{\lambda_1}(g(\gamma),\gamma,\dots,\gamma)g^{\prime}(\gamma)
+\sum_{j=2}^{n}f_{\lambda_j}(g(\gamma),\gamma,\dots,\gamma)=0,\ \forall\ \gamma\geq c_*.\end{equation}
So by \eqref{f},
\begin{equation}\label{gprime}\frac{d\beta}{d\gamma}=g^{\prime}(\gamma)=-\frac{\sum_{j=2}^{n}f_{\lambda_j}(g(\gamma),\gamma,\dots,\gamma)}
{f_{\lambda_1}(g(\gamma),\gamma,\dots,\gamma)}<0.\end{equation}
In particular, by the symmetry of $f$, we can deduce that
\begin{equation}\label{gc*}
g^{\prime}(c_*)=1-n.
\end{equation}
Let
$$F(\beta,\gamma,\dots,\gamma)=-\frac{\sum_{j=2}^{n}f_{\lambda_j}(\beta,\gamma,\dots,\gamma)}
{f_{\lambda_1}(\beta,\gamma,\dots,\gamma)}.$$
Then
$$\frac{d\beta}{d\gamma}=F(\beta,\gamma,\dots,\gamma).$$
Since $\partial F/\partial \beta$ is continuous, then by the extension theorem of ODE, $\beta=g(\gamma),\gamma>c_*$ can be extended to the left of $c_*$. And because $g^{\prime}(\gamma)<0$ and due to Lemma \ref{lem-w}$, \gamma>0$, then $\beta=g(\gamma)$ can be extended to $\gamma_0$, $0\leq \gamma_0<c_*$ such that $\lim_{\gamma\to\gamma_0^+}g(\gamma)=+\infty.$ Then the maximum existence interval of $\beta=g(\gamma)$ is $(\gamma_0,+\infty),0\leq \gamma_0<c_*$, that is,
\begin{equation}\label{g}\beta=g(\gamma),\ \gamma\in(\gamma_0,+\infty).\end{equation}
On the other hand, differentiating \eqref{g2prime} with respect to $\gamma$, we have that
$$g^{\prime\prime}(\gamma)=-\dfrac{\Lambda^T\left(\frac{\partial^2f}{\partial\lambda_i\partial\lambda_j}\right)\Lambda}{f_{\lambda_1}},$$
where $\Lambda^T=(g^{\prime}(\gamma),1,\dots,1).$ Since $f$ is concave and $f_{\lambda_1}>0$, then
\begin{equation}\label{gconvex}g^{\prime\prime}(\gamma)>0.
\end{equation}
Then
\begin{equation}\label{j2}g^{\prime}(\gamma)>g^{\prime}(c_*)=1-n\ \mbox{for}\ \gamma>c_*.
\end{equation}
Since $\lim_{\gamma\to \gamma_0^+}g(\gamma)=+\infty,$ then we declare that
\begin{equation}\label{j1}
g(\gamma)>(1-n)\gamma, \gamma_0<\gamma<+\infty.
\end{equation}
On the contrary, since $g^{\prime}(\gamma)<0$ and $g(c_*)=c_*$, then there exists some $\bar{\gamma}\in (c_*,+\infty)$ such that $g(\bar{\gamma})=(1-n)\bar{\gamma}$ and $g(\gamma)<(1-n)\gamma$ for $\gamma>\bar{\gamma}.$ So
\begin{align*}
g^{\prime}(\bar{\gamma})&=\lim_{\gamma\to \bar{\gamma}^+}\dfrac{g(\gamma)-g(\bar{\gamma})}{\gamma-\bar{\gamma}}\\
&\leq\lim_{\gamma\to \bar{\gamma}^+}\dfrac{(1-n)\gamma-(1-n)\bar{\gamma}}{\gamma-\bar{\gamma}}\\
&=1-\gamma
\end{align*}
which contradicts with \eqref{j2}. Hence \eqref{j1} holds and then $g(\gamma)+(n-1)\gamma>0$. Moreover, if $\gamma_0<\gamma<+\infty$, then $f(g(\gamma),\gamma,\dots,\gamma)=1$ and $g(\gamma)$ may be positive. So $(g(\gamma),\gamma,\dots,\gamma)\in \Gamma.$

The Lemma is proved.
\end{proof}

\begin{lemma}\label{lemW}
Let $\alpha>0$ and $g$ be the same function as \eqref{g}. Then the problem
\begin{align}\label{W}
\dfrac{dW}{dr}&=\dfrac{g(W)-W}{r},\ r>1,\\
\label{Wboun}W(1)&=\alpha
\end{align}
has a unique solution $W=W(r,\alpha)$ and
\begin{equation}\label{Wlimit}\lim_{r\to\infty}W(r,\alpha)=c_*.\end{equation}
\end{lemma}

\begin{proof}
If $W>c_*,$ then by \eqref{gprime}, we have $g(W)<g(c_*)=c_*$. Thus $g(W)-W<0$, that is, $dW/dr<0.$

If $W<c_*,$ then by \eqref{gprime}, we have $g(W)>g(c_*)=c_*$. Thus $g(W)-W>0$, that is, $dW/dr>0.$

Let $G(W,r)=\frac{g(W)-W}{r}$, then $\partial G/\partial W$ is continuous. In addition, we know that $W=c_*$ is a special solution of \eqref{W}. Hence by the existence and uniqueness theorem of solutions to the ODE equation, we know that \eqref{W} and \eqref{Wboun} has a unique solution $W=W(r,\alpha)$. Then by the extension theorem of solutions, we know that \eqref{Wlimit} holds.

The Lemma is proved.
\end{proof}

\begin{remark}
For the proof of \eqref{Wlimit}, we can also refer to Lemma 2.2 in \cite{bl-2}.
\end{remark}
\begin{lemma}\label{lemma2} Let $u\in C^1(\mathbb{R}^{n}\b B_1)\cap C^2(\mathbb{R}^n\b\o{B_1})$ be a radial solution of \eqref{ns1} and \eqref{ns2} and
$$\alpha=u^{\prime}(1).$$
Then $\lambda(D^2u)\in \Gamma$ if and only if
\begin{equation}\label{alphainterval}\sup_{r\geq 1}W^{-1}_r(\gamma_0)\leq \alpha<+\infty,\end{equation} where $W_r(\alpha)=W(r,\alpha)$ satisfies \eqref{W} and \eqref{Wboun}.
\end{lemma}

\begin{proof}
Let $u(x)=u(|x|)=u(r)\in C^1(\mathbb{R}^{n}\b B_1)\cap C^2(\mathbb{R}^n\b\o{B_1})$ be a radial solution of \eqref{ns1} and \eqref{ns2}, then the eigenvalues of the Hessian matrix $D^2u$ are
\begin{equation}\label{eigen}\lambda_1=u^{\prime\prime},\lambda_2=\cdots=\lambda_n=\dfrac{u^{\prime}}{r}.\end{equation}
So \begin{equation}\label{f-ode}f\left(u^{\prime\prime},\frac{u^{\prime}}{r},\dots,\frac{u^{\prime}}{r}\right)=1.\end{equation}
By Lemma \ref{lemma1}, we have that
$$\gamma_0<\frac{u^{\prime}}{r}<+\infty.$$
Let $W(r)=u^{\prime}(r)/r,$
then \begin{equation}\label{Wscope}\gamma_0<W(r)<+\infty,\end{equation}
and
$$u^{\prime\prime}(r)=rW^{\prime}(r)+W(r).$$
On the other hand, by \eqref{f-ode} and \eqref{g}, we know that
$$u^{\prime\prime}(r)=g\left(\frac{u^{\prime}}{r}\right)=g(W(r)),\ \gamma_0<W(r)<+\infty.$$
So $W(r)=u^{\prime}(r)/r$ satisfies \eqref{W} and \eqref{Wboun}. In the following, we denote $W(r)=W(r,\alpha)=W_r(\alpha)$.

Differentiating \eqref{W} and \eqref{Wboun} with respect to $\alpha$, we know that $V=\partial W(r,\alpha)/\partial \alpha$ satisfies
\begin{equation}
\begin{cases}
\dfrac{\partial V}{\partial r}=\left(\dfrac{g^{\prime}(W(r,\alpha))-1}{r}\right)V,\ r>1,\\
V(1)=1.
\end{cases}
\end{equation}
Then
\begin{equation}\label{partialW}\dfrac{\partial W(r,\alpha)}{\partial \alpha}=\exp\dint_1^r\dfrac{g^{\prime}(W(t,\alpha))-1}{t}dt.\end{equation}
And then $W(r,\alpha)$ is strictly increasing in $\alpha$. Next we prove that
\begin{equation}\label{Walpha}W(r,\alpha)\to +\infty,\ \mbox{as}\ \alpha\to+\infty.\end{equation}
Indeed, if $\alpha\to +\infty,$ that is, $W(1)\to+\infty,$ as the proof of Lemma \ref{lemW}, we can know that $W(r,\alpha)>c_*.$ Hence by \eqref{gconvex} and \eqref{gc*}, we obtain that
$$g^{\prime}(W(r,\alpha))>g^{\prime}(c_*)=1-n.$$
Then by \eqref{partialW}, we get that
$$\dfrac{\partial W(r,\alpha)}{\partial \alpha}>r^{-n}.$$
And thus $W(r,\alpha)>\alpha r^{-n}+W(r,0)$. So \eqref{Walpha} holds.

Since $W(r,\alpha)$ is strictly increasing in $\alpha$, then $W^{-1}_r(\alpha)$ exists, and from \eqref{Wscope}, we know that \eqref{alphainterval} holds.

On the other hand, if \eqref{alphainterval} holds, then \eqref{Wscope} holds, by Lemma \ref{lemma1}, we know that $\lambda(D^2u)\in\Gamma.$
\end{proof}

\begin{proof}[Proof of Theorem \ref{thmns1}]
Due to the Proposition 2.1 in \cite{bl-2},
\begin{equation}\label{asymp}u(x)=\dfrac{c_*}{2}|x|^2+\mu(\alpha)+b-\frac{c_*}{2}+O(|x|^{2-n}),\end{equation}
where $\mu(\alpha)$ is the same as \eqref{mualpha}. Moreover, by the Proposition 2.1 in \cite{bl-2}, $\mu(\alpha)$ is strictly increasing in $\alpha$ and $$\lim_{\alpha\to+\infty}\mu(\alpha)=+\infty.$$
Then by Lemmas \ref{lemma1} and \ref{lemma2}, we know that Theorem \ref{thmns1} is true.
\end{proof}

\noindent{\bf{\large Acknowledgements}}

Dai is supported by Shandong Provincial Natural Science Foundation (ZR2021MA054). Bao is supported by NSF of China (11871102). Wang is supported by NSF of China (11701027) and Beijing Institute of Technology Research Fund Program for Young Scholars.

(L.M. Dai)  School of Mathematics and Information Science, Weifang
 University, Weifang, 261061, P. R. China

 {\it Email address}: lmdai@wfu.edu.cn
 \vspace{3mm}

(J.G. Bao)School of Mathematical Sciences, Beijing Normal University,
Laboratory of Mathematics and Complex Systems, Ministry of
Education, Beijing, 100875, P. R. China

{\it Email address}: jgbao@bnu.edu.cn
\vspace{3mm}

(B. Wang)School of Mathematics and Statistics, Beijing Institute of Technology, Beijing, 100081, P. R. China

{\it Email address}: wangbo89630@bit.edu.cn

\end{document}